\documentclass[a4paper,12pt]{article}

\usepackage{times}
\usepackage{amsfonts}
\usepackage{amsthm}
\usepackage{amsmath}
\usepackage{epic}
\usepackage{times}

\usepackage[margin=3cm]{geometry}
\usepackage{amsthm,amsmath}
\usepackage{fullpage}

\usepackage{tikz}
\usetikzlibrary{snakes}
\usepackage{soul}

\newcommand{\bt}{\mathbf{t}}
\newcommand{\bw}{\omega}

\newcommand{\nats}{{\mathbb N}}
\newcommand{\ints}{{\mathbb Z}}
\newcommand{\reals}{{\mathbb R}}

\newcommand{\abp}{\rho}

\def\tribo{ {\bf t }}

\def\balpha{\bar\alpha}

\newtheorem{theorem}{Theorem}[section]
\newtheorem{lemma}[theorem]{Lemma}

\newtheorem{proposition}[theorem]{Proposition}
\newtheorem{definition}[theorem]{Definition}
\newtheorem{fact}[theorem]{Fact}

\theoremstyle{definition}	 

\newtheorem{remark}[theorem]{Remark}
\newtheorem{problem}{Open problem}

\newcommand{\ignore}[1]{}

\begin{document}

\title{Balance and Abelian Complexity of the {T}ribonacci word}
\author{Gw{\'e}na{\"e}l Richomme%
\footnote{Universit\'e de Picardie Jules Verne, Laboratoire MIS (Mod\'elisation, Information, Syst\`emes), 33, Rue Saint Leu, F-80039 Amiens cedex 1, FRANCE,
E-mail: \texttt{gwenael.richomme@u-picardie.fr}}
\and 
Kalle Saari%
\footnote{Department of Mathematics,
University of Turku,
FI-20014, Finland,
email: \texttt{kasaar@utu.fi}}
\and
Luca Q. Zamboni%
\footnote{
Universit\'e de Lyon, Universit\'e Lyon 1, CNRS UMR 5208 Institut Camille Jordan, B\^atiment du
Doyen Jean Braconnier, 43, blvd du 11 novembre 1918, F-69622 Villeurbanne Cedex, France, email: \texttt{zamboni@math.univ-lyon1.fr}
 \&  Reykjavik University, School of Computer Science, Kringlan 1, 103 Reykjavik, Iceland,
email: \texttt{lqz@ru.is}}
}

\maketitle

\abstract{G.~Rauzy showed that the Tribonacci minimal subshift  generated by  the morphism $\tau:$  $0\mapsto 01,$ $1\mapsto 02$ and $2\mapsto 0$  is 
measure-theoretically conjugate to an exchange of three fractal 
domains on a compact set in $\reals^2,$ each domain being translated by the same vector modulo a lattice.  In this paper we study the Abelian complexity $\abp(n)$ of the Tribonacci word $\tribo$ which is the unique fixed point of $\tau.$ We show that $\abp(n)\in\{3,4,5,6,7\}$ for each $n\geq 1,$ and that each of these five values is assumed. Our proof relies on the fact that the Tribonacci word is $2$-balanced, i.e., for all factors $U$ and $V$ of $\tribo$ of equal length, and for every letter $a\in \{0,1,2\},$ the number of occurrences of $a$ in $U$ and the number of occurrences of $a$ in $V$ differ by at most $2.$  While this result is announced in several papers, to the best of our knowledge no proof of this fact has ever been published. We offer two very different proofs of the $2$-balance property of $\tribo.$ The first uses the word combinatorial properties of the generating morphism, while the second exploits the spectral properties of the incidence matrix of $\tau.$}

\section{\label{overview}Introduction}

Given a finite non-empty set $A,$ called the {\it alphabet}, we denote by $A^*$ the free monoid generated by $A.$
The identity element of $A^*,$ called the {\it empty word}, will be denoted by 
$\varepsilon.$  For any word $u=a_1a_2\cdots a_n \in A^*,$ the length
of $u$ is the quantity $n$ and is denoted by $|u|.$ By convention, the length
of the empty word $\varepsilon$ is taken to be $0.$
For each $a\in A,$ let $|u|_a$  denote the number of occurrences of the letter $a$ in $u.$  We denote by $A^\nats$ the set of (right) infinite words on the alphabet $A.$ Given an infinite word  $\omega =\omega_0\omega_1\omega_2\cdots \in A^{\nats},$  any finite word of the form $\omega_{i}\omega_{i+1}\cdots \omega_{i+n-1}$ (with $i \geq 0$ and $n\geq 1$) is called a {\it factor} of $\omega.$ 
Let 
\[{\mathcal F}_{\omega}(n)=\{\omega_{i}\omega_{i+1}\cdots\omega_{i+n-1}\,|\,i\geq 0\}\]
denote the set of all factors of $\omega$ of length $n,$ and set $p_{\omega}(n)=\mbox{Card}({\mathcal F}_{\omega}(n)).$ The function $p_{\omega}:\nats \rightarrow \nats$ is called the {\it subword complexity function} of $\omega.$ A fundamental result due to Hedlund
and Morse \cite{MoHe1} states that a word $\omega $ is ultimately 
periodic if and only if for some $n$ the subword complexity $p_{\omega 
}(n)\leq n.$  Words of subword complexity $p(n)=n+1$ are called 
{\it Sturmian words.} The most well-known Sturmian word is the 
so-called Fibonacci word 
\[{\bf f}=01001010010010100101001001010010
010100101001001010010\cdots\]
fixed by the morphism $0\mapsto 01$ and $1\mapsto 0.$ 

Sturmian words admit various types of 
characterizations of geometric and combinatorial nature. We give two examples: 
In \cite 
{MorHed1940} Hedlund and Morse showed that each Sturmian word may be 
realized measure-theoretically by an irrational rotation on the 
circle. That is, every Sturmian word is obtained by coding the 
symbolic orbit of a point $x$ on the circle (of circumference one) 
under a rotation by an irrational angle $\alpha $ where the circle is 
partitioned
into two complementary intervals, one of length $\alpha $ and the 
other of length $1-\alpha .$
And conversely each such coding gives rise to a Sturmian word.  

Sturmian words are equally characterized by a certain balance property. We begin with the following definition: 

\begin{definition} An infinite word $\omega \in A^\nats$ is said to be  $C$-{\it balanced}  ($C$ a positive integer) if   $\bigl||U|_a-|V|_a\bigr|\leq C$ for all factors $U$ and $V$ of $\omega$ of equal length, and each $a\in A.$ We say $\omega$ is  {\it balanced} if it is $1$-balanced.
\end{definition}

\begin{theorem}[Theorem 2.1.5 in \cite{Lothaire2002}] An infinite word $\omega $
is Sturmian if and only if $\omega $ is a binary  
aperiodic balanced word.
\end{theorem}

In \cite{Ra1,Ra2}, Rauzy showed that the regular continued fraction algorithm  provides  a formidable link between the arithmetic/diophantine  
properties
of an irrational number $\alpha$, the ergodic/dynamical properties of 
a rotation by angle $\alpha$ on the circle, and the combinatorial (balance) properties of Sturmian words. 

A fundamental
problem is to generalize and extend this rich interaction to 
higher dimensions,
either by starting with a specified symbolic dynamical system, or by a  class of 
words satisfying particular combinatorial conditions. 
In case of dimension $2,$ there are a couple of different dynamical  systems  which are natural candidates: For instance, 
promising results have been obtained by the third author together with  Ferenczi and Holton
by considering the dynamics of $3$-interval exchange 
transformations on the unit
interval \cite{FeHoZa1, FeHoZa2, FeHoZa3, FeHoZa4}, and also by Arnoux, Berth\'e and Ito by considering two rotations on the
circle \cite{ArBeIt}.
However, ever since the early work of Rauzy in \cite{Ra3}, the most 
natural generalization
was thought to be the one stemming from a
rotation on the $2$-torus. In this context, the associated symbolic 
counterpart is given by a class of words of subword complexity $2n+1,$ now called {\it Arnoux-Rauzy words}, originally introduced by Arnoux and Rauzy in \cite{ArRa}. 
In this setting, the {\it Tribonacci word}
\[\tribo =\lim_{n\rightarrow \infty} \tau^{n}(0)= 01020100102010 \cdots\]
defined as the fixed point of the morphism $\tau$
\[0\mapsto 01\,\,\,\,\,\,1\mapsto 02\,\,\,\,\,\,2\mapsto 0,\]
is the natural analogue of the Fibonacci word ${\bf f}.$
In fact, in \cite{Ra3}  Rauzy  showed that the Tribonacci minimal subshift (the shift orbit closure of  $\tribo$) is a 
natural  coding of a rotation on the $2$-dimensional torus $T^2,$ i.e., is 
measure-theoretically conjugate to an exchange of three fractal 
domains on a compact set in $\reals^2,$ each domain being translated by the same vector modulo a lattice. 
In \cite{CassFerZam}, the third author together with Cassaigne and  Ferenczi showed that
there exist Arnoux-Rauzy words which are arbitrarily imbalanced, that is to say, which are not $C$-balanced for any positive integer $C,$ and that such Arnoux-Rauzy words cannot be measure-theoretically conjugate to a rotation on the $2$-torus. 

In the first part of this paper, we prove:
 \begin{theorem}\label{2bal} The Tribonacci word $\tribo$ is $2$-balanced.
\end{theorem}

While this result is announced in several papers (see
 for instance \cite{Ber1, Ber2, CassFerZam, GlenJustin2009, Vuillon2003BBMS}), to the best of our knowledge no proof of this fact has ever been published. In this paper we give two different proofs of Theorem~\ref{2bal}. The first is a proof by induction which uses the word combinatorial properties of the generating morphism $\tau.$ The second proof relies on the spectral properties of the incidence matrix and is more in the spirit of the methods developed by Adamczewski in \cite{Adam}. 
 
 Our hope is that one of the two proofs may be extended to establish a balance property for the general $m$-bonacci word ($m\geq 2$) defined as the fixed point of the morphism
 \[0\mapsto 01\,\,\,\,\,\,1\mapsto 02\,\,\,\,\,\,\cdots \,\,\,\,\,\,(m-2)\mapsto 0(m-1)\,\,\,\,\,\,(m-1)\mapsto 0.\]
Computer simulations show that the $4$-bonacci word is not $2$-balanced.
In fact, let $u$ be the factor of the $4$-bonacci word of length $3305$ occurring in position $2663$ (starting with $0$), and 
$v$ the factor of length $3305$ occurring in position $9048.$ Then $|u|_1 = 891$ and $|v|_1 = 888,$ which shows that the $4$-bonacci word is not $2$-balanced.
It is possible that in general the $m$-bonacci word is $(m-1)$-balanced as stated in
\cite{GlenJustin2009}.

We then apply Theorem~\ref{2bal} to study the so-called Abelian complexity
of the Tribonacci word. 
Following \cite{RichommeSaariZamboni2009abelian}, two words $u$ and $v$ in $A^*$ are said to be {\it Abelian equivalent,} denoted $u\sim_{\mbox{ab}} v,$  if and only if $|u|_a=|v|_a$ for all $a\in A.$ For instance, the words $ababa$ and $aaabb$ are Abelian equivalent whereas they are not equivalent to $aabbb$.  It is readily verified that $\sim_{\mbox{ab}}$ defines an equivalence relation on $A^*.$

We define \[{\mathcal F}^{\mbox{ab}}_{\omega}(n)={\mathcal F}_{\omega}(n)/\!\!\sim_{\mbox{ab}}\] and set $\abp_{\omega}=\mbox{Card}({\mathcal F}^{\mbox{ab}}_{\omega}(n)).$ 
The function
$\abp =\abp_{\omega} :\nats \rightarrow \nats$ which counts the number of pairwise non Abelian equivalent factors of $\omega$ of length $n$ is called the {\it Abelian complexity} or {\it ab-complexity} for short. 
Using  Theorem~\ref{2bal} we compute the ab-complexity of the Tribonacci word $\tribo$:
 
 \begin{theorem}\label{triboabcom} Let $\tribo$ denote the Tribonacci word. Then 
$\abp_{\tribo} (n) \in \{3,4,5,6,7\}$ for every positive integer $n.$
Moreover, each of these five values is attained.
\end{theorem}
 
We also obtain several equivalent characterizations of those values $n$ for which
$\abp(n)=3.$ As a consequence we show that $\abp(n)=3$ for infinitely many $n.$ 
Similarly we show that $\abp(n)=7$ for infinitely many $n,$ although the least $n$ for which $\abp(n)=7$ is $n=3914.$ We conclude with some open questions.

\vspace{.2in}
\noindent{\large \bf Acknowledgements:} The second author is partially supported by grant no. 8121419 from the Finnish Academy. The third author is partially supported by grant no. 090038011 from the Icelandic Research Fund.

\section{\label{S:balance}Two Proofs of Theorem~\ref{2bal}}

\subsection{Preliminaries to Proof 1 of Theorem~\ref{2bal} } 

Let $\tau: \{0,1,2\}\rightarrow \{0,1,2\}^*$ denote the morphism
\[0\mapsto 01\,\,\,\,\,\,1\mapsto 02\,\,\,\,\,\,2\mapsto 0,\]
and let \[\tribo = \lim_{n\rightarrow \infty}\tau^{n}(0)= 01020100102010 \cdots\]
denote the Tribonacci word.

For each $u\in \{0, 1, 2\}^*,$ we denote by
$\Psi(u)$ the {\it Parikh vector} associated to $u,$ that is 
\[\Psi(u)=(|u|_0,|u|_1,|u|_2).\]
For each infinite word $\omega \in \{0, 1, 2\}^{\nats}$
let $\Psi_\omega(n)$ denote the set of Parikh vectors of factors of length $n$ of  $\omega$:
\[\Psi_\omega(n)=\{\Psi(u)\,|\, u \in {\mathcal F}_{\omega}(n)\}.\]
Thus we have 
\[\abp_{\omega}(n) = \mbox{Card}(\Psi_\omega(n)).\]

We will need the following two lemmas.

\begin{lemma}\label{L0} Let $U$ be a non-empty factor of $\tribo.$ Then there exists a factor $u$ of $\tribo$ such that  
$\Psi(U)=(|u|+\delta, |u|_0, |u|_1)$ for some $\delta \in \{-1,0,1\}.$
Moreover  if  $|U|\geq 3$
then $|u|<|U|.$ 
\end{lemma}

\begin{proof} It is readily seen that every non-empty factor $U$  of $\tribo$ may be written as either $U=\tau(u),$ or $U=0^{-1}\tau(u),$ or $U=\tau(u)0,$ or $U=0^{-1}\tau(u)0$ for some factor $u$ of $\tribo.$  Moreover, 
$|\tau(u)|_0=|u|,$ $|\tau(u)|_1=|u|_0,$ and $|\tau(u)|_2=|u|_1.$
Thus we have that 
\[\Psi(U)=(|u|+\delta, |u|_0, |u|_1)\] for some $\delta \in \{-1,0,1\}.$ Moreover, from above it follows that $|U|\geq |\tau(u)|-1.$
Also, for all $u$ we have $|\tau(u)|\geq |u|,$ and if $|u|\geq 3,$ then $|\tau(u)|\geq |u|+2$ since either $u$  contains at least two occurrences of $0,$ or $u$ contains both an occurrence of  $0$ and an occurrence of $1.$
Finally, to see that $|u|<|U|$ whenever $|U|\geq 3,$ we  suppose to the contrary that  $|u|\geq |U|\geq 3;$ then 
\[|U|\geq |\tau(u)|-1\geq |u|+1 >|u|,\] which is a contradiction.

\end{proof}

\noindent{\bf Note:} In what follows, we will often apply the above lemma to a pair of words $U$ and $V$ where $|V|_0=|U|_0 +2$ and $|U|_i=|V|_i+3$ for some $i\in \{1,2\}.$ Note that under these hypotheses, $|U|\geq 7$ since $U$ contains either at least $3$ occurrences of $1,$ or at least $3$ occurrences of $2.$
Since every $1$ and $2$ occurring in $\tribo$ is always preceded by  $0,$ we deduce that $|U|_0\geq 2,$ and hence $|V|\geq 4.$
\vspace{.2in}

\begin{lemma}\label{L1} Let $U$ and $V$ be factors of $\tribo$ satisfying the following two conditions:
\begin{itemize}
\item[] $|V|_0 = |U|_0 +2$
\item[]$|U|_i= |V|_i +3$
\end{itemize}
for some $i\in \{1,2\}.$ Then there exist factors $u$ and $v$ of $\tribo,$ with $|u|<|U|,$ $|v|<|V|,$ $|u| \leq |v|,$ and $|u|_{i-1}=|v|_{i-1}+3.$
\end{lemma}

\begin{proof} By the previous lemma there exist factors $u,v$ with $|u|<|U|,$ $|v|<|V|,$ such that  

\[\Psi (U)=(|u|+\delta _1, |u|_0, |u|_1) \,\,\,\,\mbox{and}\,\,\,\,
\Psi (V)=(|v|+\delta _2, |v|_0, |v|_1)\]
for some $\delta _1,\delta _2 \in \{-1,0,1\}.$
By hypothesis we have
\[|v|+\delta_2=|V|_0=|U|_0+2=|u|+\delta _1 +2,\]
whence
\[|u|=|v| +\delta_2-\delta_1 -2 \leq |v|.\]
Finally, the condition $|U|_i= |V|_i +3$ for $i\in \{1,2\}$ implies $|u|_{i-1}=|v|_{i-1}+3.$

\end{proof}

In what follows, we will make use the following terminology: We say two finite words $U$ and $V$ on the alphabet $\{0,1,2\}$ are  {\it pairwise $2$-imbalanced} if 
$|U|=|V|$ and $\bigl|U|_i-|V|_i\bigr|\geq 3$ for some $i\in \{0,1,2\}.$

\subsection{Proof 1 of Theorem~\ref{2bal}} Suppose to the contrary that $\tribo$ is not $2$-balanced. Then there exists a shortest pair of factors $U$ and $V$ with
$|U|=|V|$ and $|U|_i-|V|_i\geq 3$ for some $i\in \{0,1,2\}.$ If $|U|_i-|V|_i>3,$ then by removing the last letter from each of $U$ and $V,$ we would obtain a shorter pair of words of equal length which are pairwise $2$-imbalanced.
Thus the minimality condition on $|U|$ implies that $|U|_i-|V|_i=3.$ Also, it is easily checked that all factors of length less or equal to $3$ are $2$-balanced whence $|U|\geq 4.$ We consider three cases: $i=0,$ then $i=2,$ and finally $i=1.$

\noindent {\bf Case 1: $|U|_0-|V|_0=3.$}  By Lemma~\ref{L0}
 there exist factors $u,v$ with $|u|<|U|,$ $|v|<|V|,$ such that  

\[\Psi (U)=(|u|+\delta _1, |u|_0, |u|_1) \,\,\,\,\mbox{and}\,\,\,\,
\Psi (V)=(|v|+\delta _2, |v|_0, |v|_1)\]
for some $\delta _1,\delta _2 \in \{-1,0,1\}.$ The condition $|U|_0=|V|_0+3$ implies that 
\[|u|+\delta _1=|v|+ \delta _2 +3\]
that is
\[|u|-|v|=(3+\delta_2 -\delta_1)>0.\]

\noindent The condition $|U|=|V|$ implies that
\[|u|+\delta_1 +|u|_0 +|u|_1=|v|+\delta _2 +|v|_0+|v|_1\]
or equivalently
\[2|u|+\delta_1-|u|_2=2|v|+\delta_2-|v|_2.\]
Thus
\[|u|_2-|v|_2=2(|u|-|v|) +\delta_1-\delta_2=2(3+\delta_2 -\delta_1)+\delta_1-\delta_2=6+\delta_2-\delta_1.\]
Let $u'$ be the prefix of $u$ of length $|u|-(3+\delta_2 -\delta_1)=|v|.$
Then, $|u'|_2 \geq |v|_2+ 3$ contradicting the minimality of $|U|.$

\vspace{.2in}

\noindent {\bf Case 2: $|U|_2-|V|_2=3.$} By Lemma~\ref{L0}
 there exist factors $u,v$ with $|u|<|U|,$ $|v|<|V|,$ such that  

\[\Psi (U)=(|u|+\delta _1, |u|_0, |u|_1) \,\,\,\,\mbox{and}\,\,\,\,
\Psi (V)=(|v|+\delta _2, |v|_0, |v|_1)\]
for some $\delta _1,\delta _2 \in \{-1,0,1\}.$ The condition $|U|_2=|V|_2+3$ implies that $|u|_1=|v|_1+3.$ If $|u|\leq |v|,$ then $u$ and the prefix of $v$ of length $|u|$ are pairwise $2$-imbalanced and of length less than $|U|,$
contradicting the minimality of $|U|.$ Thus we can suppose that
\[|u|=|v|+m\,\,\,\mbox{for some} \,\,\,m\geq 1.\]

\noindent If $m\geq 2,$ then we deduce that
\[|U|_0=|u|+\delta_1\geq |u|-1 \geq |u|-(m-1)=|v|+1\geq |v|+\delta _2=|V|_0.\] 
As $|U|_2=|V|_2+3,$  $|U|_0\geq |V|_0,$ and  $|U|=|V|,$ it follows that
$|V|_1\geq |U|_1+3,$ that is $|v|_0\geq |u|_0+3.$ But then $v$ and the prefix of $u$ of length $|v|$ are pairwise $2$-imbalanced and of length less than $|U|,$ contradicting the minimality of $|U|.$ 

Thus we can suppose $m=1,$ that is $|u|=|v|+1.$  Again, if $|U|_0\geq |V|_0,$
as above we would deduce that $|v|_0\geq |u|_0+3$ which would give rise to a contradiction. So we must have that $|U|_0<|V|_0.$ This gives
\[ |v|+1 +\delta_1=|u|+\delta_1=|U|_0<|V|_0=|v|+\delta_2\]
that is
\[1+\delta_1<\delta_2\] 
which in turn implies that
\[\delta _1=-1 \,\,\,\mbox{and}\,\,\,\delta_2=1.\]
Thus
\[\Psi (U)=(|u|-1, |u|_0, |u|_1) \,\,\,\,\mbox{and}\,\,\,\,
\Psi (V)=(|v|+1, |v|_0, |v|_1).\]

\noindent Finally, the conditions $|U|=|V|$ together with $|u|=|v|+1,$ and $|u|_1=|v|_1+3$ imply that
\begin{align*}
|v|_0&= |v|_0 +(|v|+1 +|v|_1) - (|v|+1+|v|_1)\\
&= |V| -(|v|+1 +|v|_1)\\
&= |V|-(|u| + |u|_1-3)\\
&= |V|-(|u| + |u|_1  -1 -2)\\
&= |V|- (|U|-|u|_0 -2)\\
&= |u|_0+2.
\end{align*}

In summary we have: $|u|_1=|v|_1+3$ and $|v|_0=|u|_0+2.$ Thus we can apply Lemma~\ref{L1} which gives a contradiction to the minimality of $|U|.$

\vspace{.2in}

\noindent {\bf Case 3: $|U|_1-|V|_1=3.$} By Lemma~\ref{L0}
 there exist factors $u,v$ with $|u|<|U|,$ $|v|<|V|,$ such that  

\[\Psi (U)=(|u|+\delta _1, |u|_0, |u|_1) \,\,\,\,\mbox{and}\,\,\,\,
\Psi (V)=(|v|+\delta _2, |v|_0, |v|_1)\]
for some $\delta _1,\delta _2 \in \{-1,0,1\}.$ The condition $|U|_1=|V|_1+3$ implies that $|u|_0=|v|_0+3.$

If $|u|\leq |v|,$ then $u$ and the prefix of $v$ of length $|u|$ are pairwise $2$-imbalanced and of length less than $|U|,$
contradicting the minimality of $|U|.$ Thus we can suppose that
\[|u|=|v|+m\,\,\,\mbox{for some} \,\,\,m\geq 1.\]


\noindent Next we proceed like in Case~2. If $m\geq 2,$ then we deduce that
\[|U|_0=|u|+\delta_1\geq |u|-1 \geq |u|-(m-1)=|v|+1\geq |v|+\delta _2=|V|_0.\] 
As $|U|_1=|V|_1+3,$  $|U|_0\geq |V|_0,$ and  $|U|=|V|,$ it follows that
$|V|_2\geq |U|_2+3,$ that is $|v|_1\geq |u|_1+3.$ But then $v$ and the prefix of $u$ of length $|v|$ are pairwise $2$-imbalanced and of length less than $|U|,$ contradicting the minimality of $|U|.$ 

Thus we can suppose $m=1,$ that is $|u|=|v|+1.$  Again, if $|U|_0\geq |V|_0,$
as above we would deduce that $|v|_1\geq |u|_1+3$ which would give rise to a contradiction. So we must have that $|U|_0<|V|_0.$ This gives
\[ |v|+1 +\delta_1=|u|+\delta_1=|U|_0<|V|_0=|v|+\delta_2\]
that is
\[1+\delta_1<\delta_2\] 
which in turn implies that
\[\delta _1=-1 \,\,\,\mbox{and}\,\,\,\delta_2=1.\]
Thus
\[\Psi (U)=(|u|-1, |u|_0, |u|_1) \,\,\,\,\mbox{and}\,\,\,\,
\Psi (V)=(|v|+1, |v|_0, |v|_1).\]

\noindent The conditions $|U|=|V|$ together with $|u|=|v|+1,$ and $|u|_0=|v|_0+3$ imply that
\begin{align*}
|v|_1&= |v|_1 +(|v|+1 +|v|_0) - (|v|+1+|v|_0)\\
&= |V| -(|v|+1 +|v|_0)\\
&= |V|-(|u| + |u|_0-3)\\
&= |V|-(|u| + |u|_0  -1 -2)\\
&= |V|- (|U|-|u|_1 -2)\\
&= |u|_1+2.
\end{align*}

\noindent In summary we have


\[|u|=|v|+1\,\,\,\mbox{and}\,\,\,|v|_1=|u|_1+2\,\,\,\mbox{and}\,\,\,|u|_0=|v|_0+3.\]

\noindent 
Since $|u|, |v| \geq 3$,  Lemma~\ref{L0} implies that
 there exist factors $u',v'$ with $|u'|<|u|,$ $|v'|<|v|,$ such that  

\[\Psi (u)=(|u'|+\delta '_1, |u'|_0, |u'|_1) \,\,\,\,\mbox{and}\,\,\,\,
\Psi (v)=(|v'|+\delta '_2, |v'|_0, |v'|_1)\]
for some $\delta' _1,\delta' _2 \in \{-1,0,1\}.$

As $|v|_1=|u|_1+2$ we deduce that $|v'|_0=|u'|_0+2.$ Similarly, the condition
$|u|_0=|v|_0+3$ implies that
\[ |u'|+\delta_1' =|v'|+\delta_2' +3\]
so that
\[|v'|=|u'|+\delta_1'-\delta_2'-3\leq |u'|-1.\]

Now if $|v'|\leq |u'|-2,$ then if $v''$ is any factor of the Tribonacci word of length $|u'|$
beginning in $v',$ we would have $|v''|_0\geq |v'|_0+1=|u'|_0+3$ as every factor of length
$2$ or greater contains at least one  occurrence of the letter $0.$ 

\noindent So we can assume that $|v'|=|u'|-1.$
Then
\[|v'|+1 +\delta_1'=|u'|+\delta_1'=|u|_0=|v|_0+3=|v'| +\delta_2' +3\]
that is
\[\delta_1'-\delta_2'=2,\]
which in turn implies that
\[\delta' _1=1 \,\,\,\mbox{and}\,\,\,\delta'_2=-1.\]

\noindent So
\begin{align*}
|u'|_1&=|u'|_1 +(|v'|+ |v'|_0)- (|v'|+ |v'|_0)\\
&= |u'|_1 +(|v'|+1 +1 + |v'|_0-2)- (|v'|-1 +|v'|_0 +1)\\
&= |u'|_1 +(|u'| +1 +|u'|_0) - (|v'|-1 +|v'|_0+1)\\
&= |u'|_1 +(|u|-|u'|) - ( |v|-|v'|_1 +1)\\
&= |u|-( |v|-|v'|_1 +1)\\
&= |u|-(|v|+1) +|v'|_1\\
&= |v'|_1
\end{align*}

\noindent Finally, 
\begin{align*}
|u'|_2&=|u'|_2 +(|u'|_0+|u'|_1-1) - (|u'|_0+|u'|_1-1)\\
&= |u'|-1 -(|u'|_0+|u'|_1-1)\\
&= |v'| -(|v'|_0-2 + |v'|_1 -1)\\
&= |v'|_2 +3.
\end{align*}

\noindent So in summary we have
\[|v'|_0=|u'|_0+2\,\,\,\mbox{and}\,\,\,|u'|_2=|v'|_2+3\]
to which we can apply Lemma~\ref{L1} to obtain the desired contradiction.

\subsection{Preliminaries to Proof 2 of Theorem~\ref{2bal}}

For the results in this section, see \cite[chapter 10]{Lothaire2005}.

We write  $\tribo = u_0 u_1 u_2 \cdots$ with $u_i \in \{0,1,2\}.$ 
The {\it Tribonacci numbers} $T_k$ are defined by $T_k  = |\tau^k(0)|.$ Hence
$T_k = T_{k-1} + T_{k-2} + T_{k-3}$ for $k\geq 3,$ and
\[
T_0 = 1, \quad T_1 = 2, \quad T_2 = 4, \quad T_3 = 7,  \quad T_4 = 13, \ldots
\]

\noindent Let \[
M = \left( \begin{matrix} 1 & 1 & 1 \\ 1 & 0 & 0 \\ 0 & 1 & 0 \end{matrix} \right)
\]
denote the incidence matrix of the morphism $\tau,$ i.e., $M_{ij} =|\tau(j)|_i$
for all $i,j\in \{0,1,2\}.$
The eigenvalues of $M$ are the roots of the polynomial $x^3 - x^2 - x - 1.$ We denote them $\beta, \alpha$, and $\balpha$ so that $\beta$ is real and 
 $\alpha$ and $\balpha$ are complex conjugates. (Given a complex number $z=a+bi,$ we denote by $\bar{z}=a-bi$ the complex conjugate.)
 We have
\[
\beta = 1.83928\ldots \qquad {\rm and }  \qquad |\alpha| =  |\balpha| =  0.73735\ldots 
\]

More information concerning the roots $\alpha$ and $\beta$ may be found in the comments of sequences A058265 and A000073 in ``The On-Line Encyclopedia of Integer Sequences''. In particular, some formulas relating $\alpha$ and $\balpha$ to $\beta$ are recalled.

\noindent It is well-known that the frequencies of letters in $\tribo$ exist and 
\[
{\rm Freq}_{\tribo}(0) = \frac{1}{\beta}, \quad {\rm Freq}_{\tribo}(1) = \frac{1}{\beta^2}, \quad {\rm Freq}_{\tribo}(2)  = \frac{1}{\beta^3}.
\]

\noindent Set
\[
v_{\beta} = \left(\begin{matrix}  {\beta^{-1}} \\  {\beta^{-2}} \\ {\beta^{-3} }   \end{matrix} \right), \quad 
v_{\alpha} =  \left( \begin{matrix}{\alpha^{-1}} \\ { \alpha^{-2}} \\ {\alpha^{-3}}  \end{matrix}  \right), \quad
v_{\balpha} =  \left( \begin{matrix}{\balpha^{-1}} \\ { \balpha^{-2}} \\ {\balpha^{-3}}  \end{matrix}  \right).
\]
Then $v_{\beta}, v_{\alpha}, v_{\balpha}$ are eigenvectors of $M$  corresponding to $\beta, \alpha, \balpha$, respectively, normalized so that in each case the sum of coordinates equals~$1.$
Let us  denote by $a_{\beta}, a_{\alpha}, a_{\balpha}$ the complex numbers for which 
\[ 
 a_{\beta} v_{\beta} +  a_{\alpha} v_{\alpha}  +  a_{\balpha} v_{\balpha} = \left(\begin{matrix}    1\\ 0\\ 0  \end{matrix} \right).
 \]
\noindent The numbers $a_\beta, a_\alpha$ and $a_{\bar{\alpha}}$ are unique, and therefore    
\[
a_{\balpha} = \bar{a_{\alpha}} \qquad {\rm and } \quad |a_{\alpha}| = | a_{\balpha}| = 0.14135\ldots
\]

Denote by 
\[
e_1 = \left(\begin{matrix}1\\  0\\  0 \end{matrix} \right), \quad e_2 = \left(\begin{matrix} 0\\ 1\\  0 \end{matrix} \right), \quad e_3 = \left(\begin{matrix}   0 \\ 0 \\ 1 \end{matrix} \right).
\]
Then we can represent letter frequencies as  
${\rm Freq_{\tribo}(i)} =  <\kern-3pt v_{\beta} , e_i \kern-3pt>$ for each letter $i=0,1,2$, where $<\kern-3pt\cdot , \cdot \kern-3pt>$ denotes the Hermitian scalar product. 

\subsection{Proof 2 of Theorem~\ref{2bal}}

Every natural number $N$ has a unique  {\it Zeckendorff} Tribonacci representation
\[
N = \sum_{k \geq 0} p_k T_k,
\]
where $p_k \in \{0,1\}$, all but finitely many $p_k$ equal $0,$ and  if $p_{k} = p_{k-1} = 1$, then  $p_{k-2} = 0$.

\noindent It can be shown (see \cite[Proposition 10.7.4 on page 510]{Lothaire2005}) that for each letter $i\in \{0,1,2\}$,
\[
|u_0 \cdots u_{N-1}|_i =  a_{\beta}  <\kern-3pt v_{\beta},e_i \kern-3pt > \sum_{k\geq 0} p_k \beta^k  
+ a_{\alpha} <\kern-3pt v_{\alpha},e_i \kern-3pt >  \sum_{k\geq 0} p_k \alpha^k   
+ a_{\balpha}   <\kern-3pt v_{\balpha},e_i \kern-3pt >  \sum_{k\geq 0} p_k \balpha^k  ,
\]
and 
\[
N =    a_{\beta}   \sum_{k\geq 0} p_k \beta^k  
+ a_{\alpha}   \sum_{k\geq 0} p_k \alpha^k   
+ a_{\balpha}   \sum_{k\geq 0} p_k \balpha^k.  
\]
These two identities  imply
\begin{align} \label{eq1}
|u_0 \cdots u_{N-1}|_i - N <\kern-3pt v_{\beta} , e_i\kern-3pt> &= a_{\alpha} <\kern-3pt v_{\alpha} - v_{\beta},e_i \kern-3pt >  \sum_{k\geq 0} p_k \alpha^k   
+ a_{\balpha}   <\kern-3pt v_{\balpha} -  v_{\beta},e_i \kern-3pt >  \sum_{k\geq 0} p_k \balpha^k   \cr
&= 2\, {\rm Re}\bigl(  a_{\alpha} <\kern-3pt v_{\alpha} - v_{\beta},e_i \kern-3pt >   \sum_{k\geq 0} p_k \alpha^k   \bigl), \cr
&=  \sum_{k\geq 0}  2\, p_k \,  {\rm Re}\Bigl(  a_{\alpha} \bigl( \alpha^{-(i+1)} -  \beta^{-(i+1)}   \bigr)  \alpha^k   \Bigl), 
\end{align}
where ${\rm Re}(\cdot)$ denotes the real part of a complex number. The second equality above holds because complex conjugation commutes with multiplication and addition.

The following lemma shows that 
in order to prove Theorem~\ref{2bal}, it suffices to determine adequate upper and lower bounds for the above equation~\eqref{eq1}:

\begin{lemma} \label{lemma1}
 Let $x= x_0 x_1 x_2 \cdots$ be an infinite word on a finite alphabet $A,$ and $b\in A$ a  letter occurring in $x.$ If there exist real numbers $A,B$ and $\gamma$ for which  
\[
A < |x_0 x_1 \cdots x_{N-1} |_b - N\gamma < B 
\]
for all integers $N\geq 1$, then the word $x$  is $\lfloor 2(B-A) \rfloor$-balanced with respect to the 
letter~$b,$ i.e.,
\[\bigl||U|_b - |V|_b \bigr|\leq \lfloor 2(B-A)\rfloor\] 
for all factors $U,V$ of equal length.
\end{lemma}

\begin{proof}
Let $i,j \geq 0$ be integers. Since 
\[
|x_i x_{i+1} \cdots  x_{i+N-1}|_b - N\gamma =  \Bigl(| x_0 x_{1} \cdots x_{i+N-1}|_b - (N+i) \gamma \Bigr) - \Bigl( |x_0 x_{1}  \cdots x_{i-1}|_b  -  i \gamma\Bigr),
\]
we see that 
\begin{equation}\label{eqbullet}
 A - B  <  |x_i x_{i+1} \cdots  x_{i+N-1}|_b - N\gamma < B - A.
\end{equation}
Consequently,
\[
 \Bigl|   |x_i x_{i+1} \cdots  x_{i+N-1}|_b -  |x_j x_{j+1} \cdots  x_{j+N-1}|_b \Bigr|  < (B-A) - (A-B) =2(B - A).
\]
\end{proof}

We now determine good upper and lower bounds for Eq.~\eqref{eq1}.
The numerical calculations in Propositions~\ref{prop1}, \ref{prop2} and \ref{prop3}  were verified independently by the first and second authors via different computer programs.

\begin{proposition} \label{prop1} We have
\[
-0.6 < |u_0 u_1 \cdots u_{N-1} |_0 - \frac{N}{\beta } < 0.9.
\]
\end{proposition}
\begin{proof}  By \eqref{eq1}, we have 
\[
|u_0 \cdots u_{N-1}|_0 - \frac{N}{\beta} =  \sum_{k\geq 0} 2 \, p_k \, {\rm Re}\Bigl(  a_{\alpha}(\frac{1}{\alpha} - \frac{1}{\beta})    \alpha^k   \Bigr).
\]
Set 
\[
S_L =   \sum_{0 \leq k \leq 7} 2 p_k \,   {\rm Re}\Bigl(a_{\alpha}(\frac{1}{\alpha} - \frac{1}{\beta})     \alpha^k   \Bigr), \qquad
S_R =   \sum_{k \geq 8}  2 \, p_k \, {\rm Re}\Bigl(a_{\alpha}(\frac{1}{\alpha} - \frac{1}{\beta})   \alpha^k   \Bigr).
\]
Then we have
\begin{equation}\label{eq2}
S_L - |S_R| \leq   S_L + S_R =  |u_0 \cdots u_{N-1}|_0 - \frac{N}{\beta}  \leq S_L + |S_R|. 
\end{equation}
Choosing the $p_k$ so that $p_k =1$  if and only if ${\rm Re}\bigl(a_{\alpha}(\frac{1}{\alpha} - \frac{1}{\beta})     \alpha^k   \bigl)$  is positive, and similarly for negative values,
it can be seen that 
\begin{equation}\label{eq3}
-0.42 <   S_L < 0.73. 
\end{equation}
To estimate the magnitude of $|S_R|$, we compute
\begin{equation} \label{eqa}
\left| \frac{1}{\alpha} - \frac{1}{\beta} \right| = 1.72457\ldots
\end{equation}
Thus
\begin{equation} \label{eq4}
|S_R| \leq 2 |a_{\alpha} | \cdot \left|  \frac{1}{\alpha}  - \frac{1}{\beta}   \right| \sum_{k\geq 8} |\alpha|^k = 
2 |a_{\alpha} | \cdot \left|  \frac{1}{\alpha}  - \frac{1}{\beta}   \right| \frac{ |\alpha|^8}{1 - |\alpha|}
< 0.17. 
\end{equation}
The claim follows from inequalities \eqref{eq2}, \eqref{eq3}, and \eqref{eq4}.
\end{proof}

\begin{proposition}\label{prop2} We have 
\[
-0.775 < |u_0 u_1 \cdots u_{N-1} |_1 - \frac{N}{\beta^2 } <  0.725.
\]
\end{proposition}

\begin{proof}
Analogously to the proof  of Prop.~\ref{prop1}, set 
\[
S_L =  \sum_{0 \leq k \leq 10} 2 \,p_k\, {\rm Re}\Bigl(a_{\alpha}(\frac{1}{\alpha^2} - \frac{1}{\beta^2})      \alpha^k   \Bigr), \qquad
S_R =  \sum_{k \geq 11} 2 \, p_k {\rm Re}\Bigl(a_{\alpha}(\frac{1}{\alpha^2} - \frac{1}{\beta^2})        \alpha^k   \Bigr).
\]
Again, by choosing the $p_k$ appropriately,  it can be verified that 
\[
-0.70  <   S_L < 0.65.  
\]
To find an upper  bound for $|S_R|$, we compute
\begin{equation}\label{eqb}
\left|  \frac{1}{\alpha^2} - \frac{1}{\beta^2}  \right| = 1.96298\ldots
\end{equation}
Therefore,
\[
|S_R| \leq 2 |a_{\alpha} | \cdot \left|  \frac{1}{\alpha^2}  - \frac{1}{\beta^2}   \right| \sum_{k\geq 11} |\alpha|^k 
= 2 |a_{\alpha} | \cdot \left|  \frac{1}{\alpha^2}  - \frac{1}{\beta^2}   \right| \frac{|\alpha|^{11}}{1 - |\alpha| } 
< 0.075.
\]
These inequalities imply the claim.
\end{proof}

\begin{proposition}\label{prop3} We have 
\[
-0.88 < |u_0 u_1 \cdots u_{N-1} |_2 - \frac{N}{\beta^3 } <  0.62.
\]
\end{proposition}

\begin{proof}
 Once again we set
\[
S_L =    \sum_{0 \leq k \leq 13}2 \, p_k\, {\rm Re}\Bigl(a_{\alpha}(\frac{1}{\alpha^3} - \frac{1}{\beta^3})    \alpha^k   \Bigr), \qquad
S_R=     \sum_{k \geq 14}  2 \, p_k\,  {\rm Re}\Bigl(a_{\alpha}(\frac{1}{\alpha^3} - \frac{1}{\beta^3}) \alpha^k   \Bigr).
\]
As above, it can be verified that 
\[
-0.8371  <   S_L < 0.5764.
\]
We compute 
\begin{equation} \label{eqc}
\left|  \frac{1}{\alpha^3} - \frac{1}{\beta^3}  \right| =  2.33887\ldots, 
\end{equation}
and thus
\[
|S_R| \leq 2 |a_{\alpha} | \cdot \left|  \frac{1}{\alpha^3}  - \frac{1}{\beta^3}   \right| \sum_{k\geq 14} |\alpha|^k = 
2 |a_{\alpha} | \cdot \left|  \frac{1}{\alpha^3}  - \frac{1}{\beta^3}   \right|   \frac{|\alpha |^{14}}{1 -|\alpha|} <  0.0354.
\]
The claim follows from these inequalities.
\end{proof}

Propositions \ref{prop1}--\ref{prop3} together with Lemma~\ref{lemma1} imply that $\tribo$ is $2$-balanced.

\begin{remark} None of the inequalities in the claims  of Propositions \ref{prop1}--\ref{prop3} are optimal. \end{remark}

\section{\label{S:ab-complexity}Abelian complexity}
In this section we study the Abelian complexity of the Tribonacci word and prove Theorem~\ref{triboabcom}.
We begin by recalling  the following basic fact from \cite{RichommeSaariZamboni2009abelian}: 
\begin{fact}\label{F:basic}
If an infinite word $\bw$ has two factors $u$ and $v$ of the same length $n$ for which the $i$th entry of the Parikh vector are p and p+c respectively, then for each $\ell = 0, \ldots, c$, there exists a factor $u_\ell$ of $\bw$ of length $n$  such that the $i$th entry of $\Psi(u_\ell)$ is equal to $p+\ell$.
\end{fact}

\noindent Henceforth, we denote $\abp$ the Abelian complexity of the Tribonacci word $\tribo.$

\begin{figure}\begin{center}
\includegraphics[height=6cm]{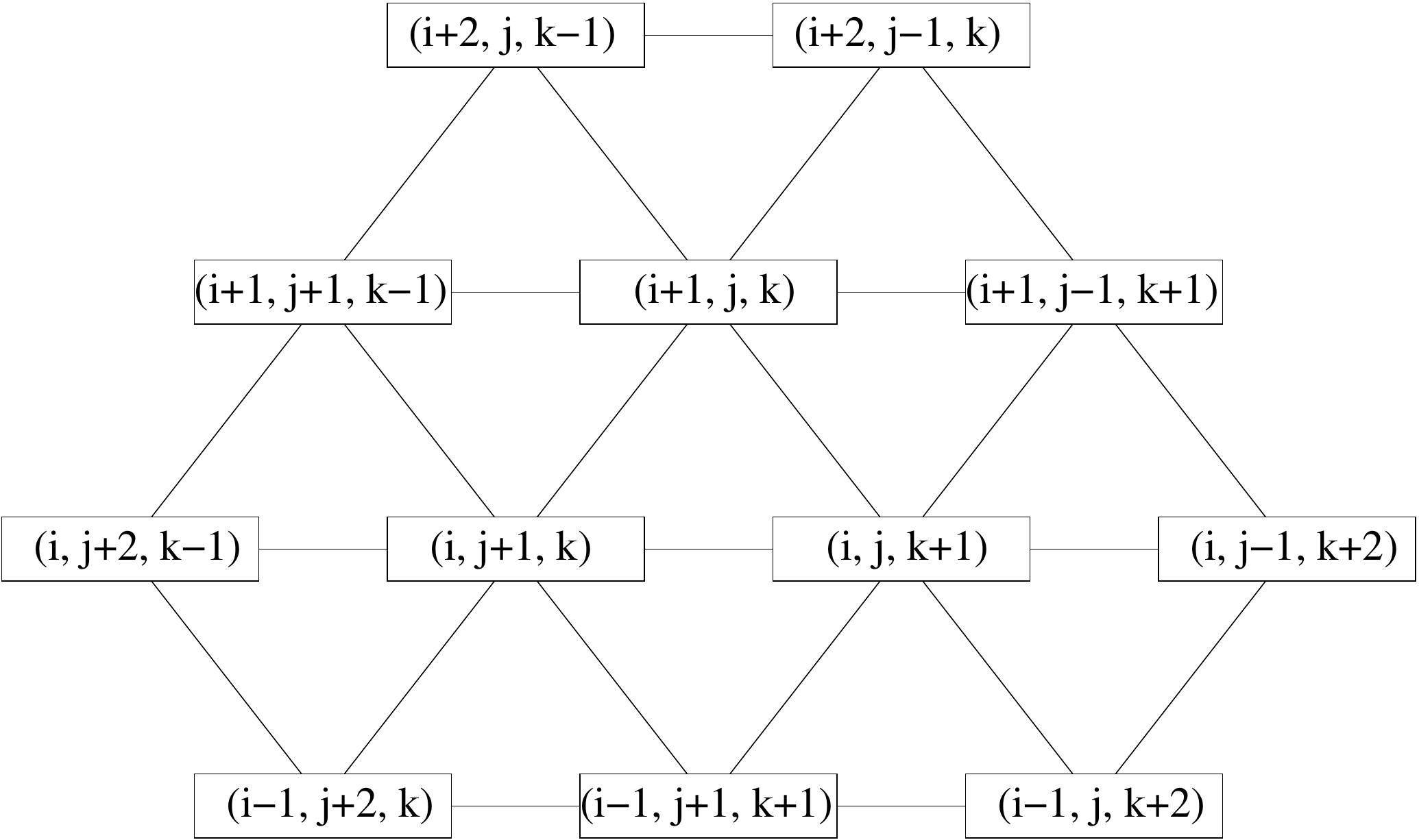}
\caption{\label{graph}Links between Parikh vectors}\end{center}
\end{figure}

\subsection{Proof of Theorem~\ref{triboabcom}}

Let us recall that a factor $u$ of an infinite word $\omega$ is \textit{right} (resp. \textit{left}) \textit{special} if there exist distinct letters $a$ and $b$ such that the words $ua$ and $ub$ (resp. $au$, $bu$) are both factors of $\omega$. A factor which is both left and right special  is called  \textit{bispecial}. 

It is well-known that for every $n \geq 1,$ $\tribo$ has exactly one right special factor of length $n-1$, and that, for this special factor that we denote $\tribo^<_{n-1}$, the three words $\tribo^<_{n-1}0, \tribo^<_{n-1}1,$ and $\tribo^<_{n-1}2$ are each factors of $\tribo$ of length $n$. Writing $\Psi(\tribo^<_{n-1})=(i,j,k),$ where $i,j,k$ are each non-negative integers, we define
$${\rm Central}(n) = \{(i+1,j,k), (i,j+1,k),(i,j,k+1)\}.$$
We have 
\begin{equation}\label{su}{\rm Central}(n) \subseteq \Psi_{\tribo}(n).\end{equation}

Given a vector $\overrightarrow{v} = (\alpha, \beta, \gamma)$, set $||\overrightarrow{v}|| = 
\max(|\alpha|, |\beta|, |\gamma|)$. The set of vectors $\overrightarrow{v}$ such that $||\overrightarrow{v}  - \overrightarrow{u}|| \leq 2$ for all $\overrightarrow{u}$ in ${\rm Central}(n)$ is given by the graph in Figure~\ref{graph}    (whose vertices are vectors, and  where two vertices $\overrightarrow{u}, \overrightarrow{v}$ are joined by an undirected edge if and only if   $||\overrightarrow{v}  - \overrightarrow{u}|| = 1$).

Since $\tribo$ is 2-balanced, $\Psi_\tribo(n)$ is a subset of this set of twelve vectors. Moreover for the same reason, we should have $||\overrightarrow{v}  - \overrightarrow{u}|| \leq 2$ for all $\overrightarrow{u}$, $\overrightarrow{v}$ in $\Psi_\tribo(n)$. This implies that the only possibility for $\Psi_\tribo(n)$ is to be a subset of one of the three sets delimited by a regular hexagon in Figure~\ref{graph}, or one of the three sets delimited by an equilateral triangle of base length 2. These sets have cardinalities 7 and 6 respectively showing that $\abp_\omega(n) \leq 7$.
\,\,\,\,\,\,\,\,\,$\Box$

\medskip

\begin{remark}By computer simulation we find that
\[(\abp(n))_{n\geq 1}=3 3 4 3 4 4 4 3 4 4 4 4 4 4 3 4 4 4 4 4 4 4 4 4 4 4 4 3 4 5 5 4 4 4 4 4 5 5 4 4 4 4\ldots\]
In particular, the least $n$ for which $\abp(n)=5$ is for $n=30.$ We also found that the smallest $n$ for which $\abp(n)=6$ is
$n=342,$ and the smallest $n$ for which  $\abp(n)=7$ is $n=3914.$ The next four values of $n$ for which $\abp(n)=7$ are $n= 4063, 4841,  4990, 7199.$ 
\end{remark}

\subsection{More on the Abelian complexity of the Tribonacci Word}

 In this section, we characterize those $n$ for which $\abp_{\tribo}(n)=3.$

We continue to adopt the notation  $\tribo^<_{n-1}$, $\Psi(\tribo^<_{n-1})$, ${\rm Central}(n)$,  and $||\overrightarrow{v}||$ introduced in the proof of Theorem~\ref{triboabcom}.
Let $B(n)$ be the set of vectors $\overrightarrow{v}$ for which there is exactly one $\overrightarrow{u}$ in ${\rm Central}(n)$ with 
$||\overrightarrow{v}  - \overrightarrow{u}|| = 2$:
\begin{equation}\label{B(n)}B(n)=\{(i-1,j+1,k+1),(i+1,j-1,k+1),(i+1,j+1,k-1)\}.\end{equation}

\begin{proposition}\label{equivalent} The following are equivalent:
\begin{enumerate} 
\item[(1)] For all factors $v$ and $w$ of $\tribo$ of length $n,$ we have
$\bigl||v|_a-|w|_a\bigr|\leq 1$ for all $a\in \{0,1,2\}.$ 
\item[(2)] $\abp_{\tribo}(n)=3.$ 
\item[(3)] $\Psi_{\tribo}(n)\cap B(n)=\emptyset.$
\item[(4)] $\tribo$ contains a bispecial special factor of length $n-1.$
\item[(5)] $n=1,$ or $n=\frac12(T_m+T_{m+2} -1)$ for $m\geq 0,$ where $(T_m)_{m\geq 0}=1,2,4,7,13,24,\ldots$ denotes the sequence of Tribonacci numbers.
\end{enumerate}
\end{proposition}

\begin{proof} Equivalence $(1)\Leftrightarrow (2)$ is an immediate consequence of Eq.~\eqref{su}.
To see the equivalence between items (4) and (5),  we note that the bispecial factors of $\tribo$ are precisely the palindromic prefixes of $\tribo$ (see \cite{RiZa}). The lengths of these words are known to be 
\[\frac{T_m+(T_{m-1}+T_m) + (T_{m-2}+T_{m-1}+T_m)-3}2=\frac{T_m+T_{m+2}-3}2\]
(see Corollary 3.11 in \cite{deLuZam} or Corollary 3.3 in \cite{WozZam}).
It follows that $(4)\Leftrightarrow (5).$ Thus it remains to prove that items (2), (3), and (4) are equivalent. We will make use of the following lemmas.

\begin{lemma}\label{imb} Let $n\geq 1,$ and let $v$ be a factor of $\tribo .$ Put
$\Psi(\tribo^<_{n-1})=(i,j,k),$ and $\Psi(v)=(i',j',k').$ Suppose either
\begin{itemize}
\item $|v|\geq n$ and either $i'\leq i-1$ or $j'\leq j-1$
\end{itemize}
or
\begin{itemize}
\item $|v|\leq n,$ and either $i'\geq i+2$ or $j'\geq j+2.$
\end{itemize}
Then $\abp(n)>3.$
\end{lemma}

\begin{proof} The first condition implies the existence of a factor
$v'$ of $\tribo$ of length $n$ with either $|v'|_0\leq i-1$ or
$|v'|_1\leq j-1.$ By  \eqref{su} it follows that the factors of $\tribo$ of length $n$ are not balanced. Similarly, the second condition implies the existence of a factor
$v'$ of $\tribo$ of length $n$ with either $|v'|_0\geq i+2$ or
$|v'|_1\geq j+2.$ Again by  \eqref{su} it follows that the factors of $\tribo$ of length $n$ are not balanced.
\end{proof}

\noindent For each $n\geq 1,$ let $\phi(n)=|\tau(\tribo^<_{n-1})0|+1.$

\begin{lemma}\label{phi} If $\abp(n)=3,$ then $\abp(\phi(n))=3.$
\end{lemma}

\begin{proof} Assume $\abp(n)=3,$ and set $N=\phi(n).$ Let $\Psi(\tribo^<_{n-1})=(i,j,k),$ and
$\tribo^<_{N-1}=(I,J,K).$  Since $\tribo^<_{N-1}=\tau(\tribo^<_{n-1})0,$ it follows that 
\begin{equation}\label{equi}I=n;\,\,\,\, J=i;\,\,\,\, K=j.\end{equation}
Now we have 
\[\Psi_{\tribo}(n)=\{(i+1,j,k),(i,j+1,k),(i,j,k+1)\}\]
and
\[\{(I+1,J,K),(I,J+1,K),(I,J,K+1)\}\subseteq \Psi_{\tribo}(N)\]
and we must show that in fact we have equality in the previous containment.
By Theorem~\ref{2bal}, the only other possible Parikh vectors of factors of $\bt$ of length $N$ are 
one of the following six vectors
\[\left(\begin{matrix} I-1\\J\\K+2\end{matrix}\right); \left(\begin{matrix} I-1\\J+2\\K\end{matrix}\right);\left(\begin{matrix} I\\J-1\\K+2\end{matrix}\right);\left(\begin{matrix} I\\J+2\\K-1\end{matrix}\right);\left(\begin{matrix} I+2\\J-1\\K\end{matrix}\right);\left(\begin{matrix} I+2\\J\\K-1\end{matrix}\right)\]
or one of the following three vectors
\[\left(\begin{matrix} I+1\\J+1\\K-1\end{matrix}\right);\left(\begin{matrix} I+1\\J-1\\K+1\end{matrix}\right);\left(\begin{matrix} I-1\\J+1\\K+1\end{matrix}\right).\]
Using Lemma~\ref{imb} we will show that if any one of these nine vectors belong to $\Psi_\bt(n)$, 
then $\abp(n)>3$, a contradiction. Hence $\abp(N)=3.$ We proceed one vector at a time.

\vspace{.1 in}

\noindent $\diamond$ Suppose $(I-1,J,K+2)\in \Psi_{\tribo}(N).$ By Lemma~\ref{L0} there exists a factor $v$ of $\tribo$ with $\Psi(v)=(i',j',k')$ and $\delta \in \{-1,0,1\}$  such that
\[
\begin{array}{rl}
I-1&=|v|+\delta \\
J&=i'\\
K+2&=j'.
\end{array}
\]

It follows from \eqref{equi} that $|v|=n-1-\delta \leq n,$ and $j'= j+2$ which contradicts Lemma~\ref{imb}.

\vspace{.1 in}

\noindent $\diamond$ Suppose $(I-1,J+2,K)\in \Psi_{\tribo}(N).$ By Lemma~\ref{L0} there exists a factor $v$ of $\tribo$ with $\Psi(v)=(i',j',k')$ and $\delta \in \{-1,0,1\}$  such that
\[
\begin{array}{rl}
I-1&=|v|+\delta \\
J+2&=i'\\
K&=j'.
\end{array}
\]

It follows from \eqref{equi} that $|v|=n-1-\delta \leq n,$ and $i'= i+2$ which contradicts Lemma~\ref{imb}.

\vspace{.1 in}

\noindent $\diamond$ Suppose $(I,J-1,K+2)\in \Psi_{\tribo}(N).$ By Lemma~\ref{L0} there exists a factor $v$ of $\tribo$ with $\Psi(v)=(i',j',k')$ and $\delta \in \{-1,0,1\}$  such that
\[
\begin{array}{rl}
I&=|v|+\delta \\
J-1&=i'\\
K+2&=j'.
\end{array}
\]

It follows from \eqref{equi} that $|v|=n-\delta \leq n+1,$ $i'= i-1,$ and $j'=j+2$ which contradicts Lemma~\ref{imb}.

\vspace{.1 in}

\noindent $\diamond$ Suppose $(I,J+2,K-1)\in \Psi_{\tribo}(N).$ By Lemma~\ref{L0} there exists a factor $v$ of $\tribo$ with $\Psi(v)=(i',j',k')$ and $\delta \in \{-1,0,1\}$  such that
\[
\begin{array}{rl}
I&=|v|+\delta \\
J+2&=i'\\
K-1&=j'.
\end{array}
\]

It follows from \eqref{equi} that $|v|=n-\delta \leq n+1,$  $i'= i+2,$ and $j'=j-1$ which contradicts Lemma~\ref{imb}.

\vspace{.1 in}

\noindent $\diamond$ Suppose $(I+2,J-1,K)\in \Psi_{\tribo}(N).$ By Lemma~\ref{L0} there exists a factor $v$ of $\tribo$ with $\Psi(v)=(i',j',k')$ and $\delta \in \{-1,0,1\}$  such that
\[
\begin{array}{rl}
I+2&=|v|+\delta \\
J-1&=i'\\
K&=j'.
\end{array}
\]

It follows from \eqref{equi} that $|v|=n+2-\delta \geq n+1,$ and $i'= i-1$ which contradicts Lemma~\ref{imb}.

\vspace{.1 in}

\noindent $\diamond$ Suppose $(I+2,J,K-1)\in \Psi_{\tribo}(N).$ By Lemma~\ref{L0} there exists a factor $v$ of $\tribo$ with $\Psi(v)=(i',j',k')$ and $\delta \in \{-1,0,1\}$  such that
\[
\begin{array}{rl}
I+2&=|v|+\delta \\
J&=i'\\
K-1&=j'.
\end{array}
\]

It follows from \eqref{equi} that $|v|=n+2-\delta \geq n+1,$ and $j'= j-1$ which contradicts Lemma~\ref{imb}.

\noindent This concludes the first six cases. Now we consider the last three:

\vspace{.1 in}

\vspace{.1 in}

\noindent $\diamond$ Suppose $(I+1,J+1,K-1)\in \Psi_{\tribo}(N).$ By Lemma~\ref{L0} there exists a factor $v$ of $\tribo$ with $\Psi(v)=(i',j',k')$ and $\delta \in \{-1,0,1\}$  such that
\[
\begin{array}{rl}
I+1&=|v|+\delta \\
J+1&=i'\\
K-1&=j'.
\end{array}
\]

It follows from \eqref{equi} that $|v|=n+1-\delta \geq n,$ and $j'= j-1$ which contradicts Lemma~\ref{imb}.

\noindent $\diamond$ Suppose $(I+1,J-1,K+1)\in \Psi_{\tribo}(N).$ By Lemma~\ref{L0} there exists a factor $v$ of $\tribo$ with $\Psi(v)=(i',j',k')$ and $\delta \in \{-1,0,1\}$  such that
\[
\begin{array}{rl}
I+1&=|v|+\delta \\
J-1&=i'\\
K+1&=j'.
\end{array}
\]

It follows from \eqref{equi} that $|v|=n+1-\delta \geq n,$ and $i'= i-1$ which contradicts Lemma~\ref{imb}.

\vspace{.1 in}

\noindent $\diamond$ Suppose $(I-1,J+1,K+1)\in \Psi_{\tribo}(N).$ By Lemma~\ref{L0} there exists a factor $v$ of $\tribo$ with $\Psi(v)=(i',j',k')$ and $\delta \in \{-1,0,1\}$  such that
\[
\begin{array}{rl}
I-1&=|v|+\delta \\
J+1&=i'\\
K+1&=j'.
\end{array}
\]

It follows from \eqref{equi} that $|v|=i'+j'+k'=n-1-\delta =i+j+k-\delta,$ $i'=i+1,$ and $j'= j+1.$ 
Thus we have $k'=k-2-\delta.$ If $\delta =-1,$ then $(i+1,j+1,k-1)\in \Psi_{\tribo}(n)$ contradicting that $\abp(n)=3.$  If $\delta =0,$ then 
$|v|= n-1$ and $k'=k-2.$ Thus any factor of $\tribo$ of length $n$ beginning
in $v$ will have at most $k-1$ many $2$'s, contradicting that the balance between any two factors of length $n$ is at most one (there is a factor of length $n$ with Parikh vector $(i, j, k+1)$). If $\delta =1,$ then $|v|=n-2$ and $k'=k-3.$ Thus any factor of length $n$ beginning in $v$ will have at most $k-2$ many $2$'s,  contradicting Theorem~\ref{2bal}.

\noindent This concludes the proof of the lemma.
\end{proof}

\begin{lemma} If $\Psi_{\tribo}(n)\cap B(n)\neq \emptyset,$ then $\Psi_{\tribo}(\phi(n))\cap B(\phi(n))\neq \emptyset.$
\end{lemma}

\begin{proof} Since $\Psi(\tribo^<_{n-1})=(i,j,k),$ we have that $\phi(n)=n+i+j +1.$  Following \eqref{B(n)},  if $\Psi (v)=(i-1,j+1,k+1)$ for some factor $v$ of $\tribo$ of length $n,$ then
$\Psi(\tau(v)0)=(n+1, i-1,j+1)\in \Psi_{\tribo}(\phi(n))\cap B(\phi(n)).$
If  $\Psi (v)=(i+1,j-1,k+1)$ for some factor $v$ of $\tribo$ of length $n,$ then
$\Psi(\tau(v)0)=(n+1, i+1,j-1)\in \Psi_{\tribo}(\phi(n))\cap B(\phi(n)).$
Finally if $\Psi (v)=(i+1,j+1,k-1)$ for some factor $v$ of $\tribo$ of length $n,$ then
$\Psi(0^{-1}\tau(v))=(n-1, i+1,j+1)\in \Psi_{\tribo}(\phi(n))\cap B(\phi(n)).$
\end{proof}

\begin{lemma} If $\Psi_{\tribo}(n)\cap B(n)= \emptyset,$ then $\tribo^<_{n-1}$ is 
a bispecial factor of $\tribo.$\end{lemma}

\begin{proof} We proceed by induction on $n.$ The result is clear for $n=1,2.$
Let $N>2,$ and suppose the result is true for all $n<N.$ Assume
$\Psi_{\tribo}(N)\cap B(N)= \emptyset .$ Set $\Psi (\tribo ^<_{N-1})=(I,J,K).$ Since $\tribo ^<_{N-1}$ is right special, it ends to  $0.$ We now show that $\tribo ^<_{N-1}$ also begins in $0.$ Suppose $\tribo ^<_{N-1}$ begins in $1.$ Then, $1^{-1}\tribo ^<_{N-1}20$ is a factor of $\tribo$ of length $N,$ and $\Psi (1^{-1}\tribo ^<_{N-1}20)=(I+1,J-1,K+1)\in \Psi_{\tribo}(N)\cap B(N)$ contrary to our assumption.  Similarly, assume $\tribo ^<_{N-1}$ begins in $2.$ Then, $2^{-1}\tribo ^<_{N-1}10$ is a factor of $\tribo$ of length $N,$ and $\Psi (2^{-1}\tribo ^<_{N-1}10)=(I+1,J+1,K-1)\in \Psi_{\tribo}(N)\cap B(N)$ contrary to our assumption. 

Having established that $\tribo ^<_{N-1}$ begins in $0,$ we can write $\tribo ^<_{N-1}=\tau(u)0$ for some factor $u$ of $\tribo.$ Also, as $\tribo^<_{N-1}$ is right special, so is $u$ and hence $u=\tribo^<_{n-1}$ for some $n,$ and so $N=\phi(n).$  It follows from the previous lemma that $\Psi_{\tribo}(n)\cap B(n) =\emptyset .$ Hence by induction hypothesis we have
that $\tribo^<_{n-1}$ is  bispecial, and hence so is $\tau(\tribo^<_{n-1})0=\tribo^<_{N-1}.$

 \end{proof}
 
 We are now ready to establish the remaining equivalences in \ref{equivalent}:
 We will show that $(3)\Rightarrow (4)\Rightarrow (2)\Rightarrow (3).$ The previous lemma states precisely that $(3)\Rightarrow (4).$ 
 That $(2)\Rightarrow (3)$ is clear from \eqref{su} and \eqref{B(n)}.
 Finally to see that $(4)\Rightarrow (2),$ we proceed by induction on $n.$ 
 The result is clear for $n=1,2.$ Let $N>2$ and suppose the result is true for all $n<N.$ Assume $\tribo^<_{N-1}$ is  bispecial. Then as $\tribo^<_{N-1}$ begins and ends in $0,$ we can write $\tribo^<_{N-1}=\tau(u)0$ for some factor $u$ of $\tribo.$ Moreover as $\tribo^<_{N-1}$ is bispecial, so is $u.$ It follows that $u=\tribo^<_{n-1}$ for some $n,$ and hence $N=\phi(n).$ By induction hypothesis we have that $\abp(n)=3.$ It now follows from Lemma~\ref{phi} that $\abp(N)=3$ as required. This completes our proof of Proposition~\ref{equivalent}.
\end{proof}

As an immediate consequence of  Proposition~\ref{equivalent}, we have that
$\abp(n)=3$ for infinitely many values of $n.$ We next show that the same is true for $\abp(n)=7.$

\begin{proposition}
\label{L:value7}
The Abelian complexity of the Tribonacci word attains the value $7$ infinitely often.
\end{proposition}

\begin{proof}
Note that $\tribo$ begins with the square $(0102010)^2 = \tau^3(00)$  and so with the squares $\tau^{n+3}(00)$. 
Consider any integer $n \geq 0$ for which all factors of length $m = 3914$ occur in $\tau^{n+3}(0)$ (we recall that a computer computation showed that $\abp(3914) = 7$).
For any factor $y$ of length $m$ there exists a conjugate $x$ of $\tau^{n+3}(0)$ such that $xy$ is a factor of $\tau^{n+3}(00)$. Since $\Psi(x) = \Psi(\tau^{n+3}(0))$ and since $\Psi(xy) = \Psi(x) + \Psi(y)$, we deduce that $\mbox{Card}\Psi_\tribo(|\tau^{n+3}(0)|+m) \geq \mbox{Card}\Psi_\tribo(m) = 7$. Since the Abelian complexity of the Tribonacci word is bounded by $7$, we get 
$\mbox{Card}\Psi_\tribo(|\tau^{n+3}(0)|+m)  = 7$.
\end{proof}

\section{\label{S:conclusion}Conclusion}

Various aspects of the Abelian complexity of the Tribonacci word remain a mystery. For instance, it is surprising to us that the value $\abp_{\tribo}(n)=7$ does not occur until $n=3914,$ but then re-occurs relatively shortly thereafter. As another example, it is verified that for all $n\leq 184,$ if $U$ and $V$ are factors of $\tribo$ of length $n,$ with $U$ a prefix of $\tribo,$ then $||U|_a-|V|_a|\leq 1,$ for all $a\in \{0,1,2\}.$ But then this property fails for $n=185.$
We list a few related open questions:

\begin{problem}
Does the Abelian complexity of the Tribonacci word attain each value in $\{4, 5, 6\}$ infinitely often?
\end{problem}

\begin{problem}
For each value $m\in \{4, 5, 6, 7\},$ characterize those $n$ for which $\abp(n)=m.$
\end{problem}

Another problem concerns the $m$-bonacci word, the generalization of the Tribonacci word (and of the celebrated Fibonacci word) to an alphabet on $m$ letters. This word is defined as the fixed point of the morphism $\tau_m: \{0, \ldots, m-1\}^*\to \{0, \ldots, m-1\}^*$ defined by $\tau_m(i) =0(i+1)$ for $0 \leq i \leq m-2$ and $\tau(m-1) = 0$. 

\begin{problem}
Prove or disprove that the $m$-bonacci word is $(m-1)$-balanced.
\end{problem}

\bibliographystyle{plain}
\bibliography{RSZshort}

\end{document}